\documentclass[12pt,a4paper]{article}
\usepackage{amsfonts}
\usepackage{amsmath}
\usepackage{amssymb}
\usepackage{fullpage}
\usepackage{amsthm}
\usepackage{graphicx}

\newtheorem{theorem}{Theorem}[section]
\newtheorem{proposition}[theorem]{Proposition}
\newtheorem{lemma}[theorem]{Lemma} 
\newtheorem{corollary}[theorem]{Corollary}

\newcommand{\nn}{{\mathbb N}}
\newcommand{\zz}{{\mathbb Z}}
\newcommand{\rr}{{\mathbb R}}
\newcommand{\ttt}{{\mathbb T}}
\newcommand{\semi}{{\rtimes}}
\newcommand{\link}{{\mathrm{Lk}}}
\newcommand{\downw}{{\downarrow}}

\title{Uncountably many quasi-isometry classes of groups of type $FP$} 

\author{Robert P. Kropholler \and 
Ian J. Leary\thanks{Partially supported by a Research Fellowship 
from the Leverhulme Trust.}  \and Ignat Soroko\thanks{Partially 
supported by NSF Grants~DMS-1107452, 1107263, 1107367 and 
`RNMS:Geometric Structures and Representation Varieties' (the GEAR 
Network).} 
}

\date{\today}

\begin{document} 

\maketitle

\begin{abstract} 
In~\cite{ufp} one of the authors constructed uncountable families
of groups of type $FP$ and of $n$-dimensional 
Poincar\'e duality groups for each $n\geq 4$.  We show that 
the groups constructed in~\cite{ufp} 
comprise uncountably many quasi-isometry classes.  We deduce 
that for each $n\geq 4$ there are uncountably many quasi-isometry
classes of acyclic $n$-manifolds admitting free cocompact 
properly discontinuous discrete group actions.  
\end{abstract}

\section{Introduction} 

\footnotetext{This work was started at MSRI, Berkeley (during 
the program {\sl Geometric Group Theory}), where 
research is supported by the National Science Foundation under
Grant No.~DMS-1440140.  It was completed at INI, Cambridge, 
during the programme {\sl Nonpositive curvature, group actions
and cohomology}, where research is supported by EPSRC grant EP/K032208/1.
The authors thank both MSRI~and~INI for their support and hospitality.}  
Throughout this article, the phrase `continuously many' will be used
to describe sets having the cardinality of the real numbers.
In~\cite{ufp} one of the authors exhibited continuously many
isomorphism types of groups of type $FP$, extending the work 
of Bestvina and Brady~\cite{bb}, who constructed the first examples 
of groups of type $FP$ that are not finitely presented.  We extend 
these results still further, by showing that the groups constructed 
in~\cite{ufp} fall into continuously many quasi-isometry classes.  

Bestvina-Brady associate a group $BB_L$ to each finite flag complex 
in such a way that the homological properties of the group $BB_L$ are 
controlled by those of the flag complex~$L$.  In particular, in the 
case when $L$ is acyclic but not contractible, $BB_L$ is type~$FP$ but
not finitely presented.  In~\cite{ufp}, a group 
$G_L(S)$ is associated to each connected finite flag
complex $L$ and each set $S\subseteq \zz$ in such a way that the 
homological properties of $G_L(S)$ are controlled by those of 
$L$ and its universal cover, $\widetilde L$.  In the case when 
$L$ and $\widetilde L$ are both acyclic, each $G_L(S)$ is 
type $FP$.  The construction of $G_L(S)$ generalizes that of 
$BB_L$, and in particular $G_L(\zz)$ is $BB_L$.  
Our first main theorem is as follows.  

\begin{theorem} \label{thm:main}
For each fixed finite connected flag complex $L$ that is not
simply-connected, there are continuously many quasi-isometry 
classes of groups~$G_L(S)$.  
\end{theorem} 

The invariant that we rely on to distinguish the groups $G_L(S)$ 
is the invariant that was introduced by Bowditch~\cite{bhb} in 
his construction of continuously many quasi-isometry classes of 
2-generator groups.  (Grigorchuk's construction of such a family,   
using growth rate to distinguish the groups~\cite{grig}, is no 
help to us because the groups $G_L(S)$ all have exponential growth.)  
Our account of Bowditch's invariant differs from his and may 
be of independent interest.  

To a graph $\Gamma$ Bowditch associates a 
set $H(\Gamma)$ of natural numbers, consisting of the lengths 
of loops in $\Gamma$ that are {\sl taut} in the sense that they 
are not consequences of shorter loops.  He describes the 
relationship between $H(\Gamma)$ and $H(\Gamma')$ in the 
case when $\Gamma$~and~$\Gamma'$ are 
quasi-isometric.  When $\Gamma$ is the Cayley graph 
associated to a group presentation satisfying the $C'(1/6)$ 
small cancellation condition, the set $H(\Gamma)$ is equal 
to the set of lengths of the relators of the presentation.  

Our proof involves estimating the set $H(\Gamma(S))$, where $\Gamma(S)$ is
the Cayley graph associated to the natural generating set for
$G_L(S)$.  The natural presentation for $G_L(S)$ 
contains relators whose lengths are parametrized by the absolute values
of the members of $S$, but it also contains many relators of length~3,
and does not satisfy the $C'(1/6)$ condition.  To apply Bowditch's 
technique we need a lower bound for the word lengths of elements in 
the kernel of the map $G_L(S)\rightarrow G_L(T)$ for $S\subseteq T$, 
in terms of $T-S$.  The Cayley graph $\Gamma(S)$ embeds naturally in 
a CAT(0) cubical complex.  Our 
lower bound on word length uses this embedding and an easy lemma 
concerning maps between CAT(0) spaces, which will be proved in 
Section~\ref{sec:three}.  In the statement, the 
singular set for a map consists of all points at which it is not
a local isometry.  

\begin{lemma} \label{lem:cat0}
Let $f:X\rightarrow Y$ be a continuous map of CAT(0) 
metric spaces, and suppose that $x\neq x'$ but $f(x)=f(x')$.  
Then the distance $d_X(x,x')$ is at least the sum of the 
distances from $x$~and~$x'$ to the singular set for $f$.  
\end{lemma} 

The so-called `Davis trick'~\cite{davis,davbook} allows one to embed
groups of type~$FP$ as retracts of Poincar\'e duality groups, and
enabled the construction of continuously many isomorphism types of
Poincar\'e dualtiy groups~\cite[Thm.~18.1]{ufp}.  In the final section
we strengthen this result in two ways.  

\begin{corollary}\label{cor:manypdn}
For each $n\geq 4$ there are continuously many quasi-isometry 
classes of non-finitely presented $n$-dimensional Poincar\'e 
duality groups.  
\end{corollary}

\begin{corollary}\label{cor:manymani} 
For each $n\geq 4$ there is a closed aspherical $n$-manifold 
admitting continuously many quasi-isometry classes of regular 
acyclic covers.  
\end{corollary} 

To establish these results, we study the behaviour of the 
Bowditch length spectrum under some semi-direct product 
constructions that arise when implementing the Davis trick. 
In contrast to the geometric methods used throughout the 
rest of the article, the proof of Theorem~\ref{thm:semiker}, 
which is the main result of the final section, 
is purely algebraic.  It would be interesting to have a 
geometric proof of this theorem and conversely 
to have algebraic proofs of our results concerning $G_L(S)$.  

\section*{Acknowledgements} 

The authors thank Noel Brady and Martin Bridson for their 
interest in this project and for stimulating discussions 
about it.  We also thank them both for having posed the 
question of whether the groups $G_L(S)$ fall into 
continuously many quasi-isometry classes, which 
initiated the project.  Finally, we thank a referee, 
whose careful reading of an earlier version of this 
article led to various improvements.

\section{Background} 

There are various types of Cayley graphs, but we shall need just one
type which we shall call the {\sl simplicial Cayley graph}
$\Gamma(G,S)$ associated to the group $G$ and generating set~$S$.
This is the simplicial graph with vertex set $G$ and edge set the
2-element sets of the form $\{g,gs\}$ for some $s\in S$.  This
definition could be made for any $S\subseteq G$; the fact that $S$
generates $G$ is equivalent to the graph $\Gamma(G,S)$ being
connected.  Any simplicial graph with a free, transitive $G$-action on
its vertex set is isomorphic as a graph with $G$-action to
$\Gamma(G,S)$ for some $S$.  The action of $G$ on the edges of
$\Gamma(G,S)$ is free if and only if $S$ contains no element of order
two.

Next we recall some material from~\cite{bhb} concerning
quasi-isometries and Bowditch's taut loop length spectrum.  Bowditch's
article~\cite{bhb} does not mention homotopies, 2-complexes or the
fundamental group, all of which play crucial roles in our account of
his work.  We believe that some readers will benefit from our 
different but equivalent account.  For this reason we restate some of
his results in our terms, and encourage the interested
reader to try to prove them before consulting~\cite{bhb}.

Each graph $\Gamma$ that we consider will be connected, simplicial, 
and will be viewed as a metric space via the path metric $d_\Gamma$, 
in which each edge has length one.  The induced metric on the vertex 
set of a Cayley graph $\Gamma(G,S)$ is thus the $S$-word length 
metric on the group $G$.  For $k>0$ an integer, recall that a 
function $f:X\rightarrow Y$ between metric spaces is {\sl $k$-Lipschitz} if 
$d_Y(f(x),f(x'))\leq k.d_X(x,x')$ for all $x,x'\in X$.  Following 
Bowditch~\cite{bhb}, we say that graphs $\Gamma$ and $\Lambda$ 
are {\sl $k$-quasi-isometric} if there exist a pair of 
$k$-Lipschitz maps of vertex sets $\phi:V(\Gamma)\rightarrow 
V(\Lambda)$ and $\psi:V(\Lambda)\rightarrow V(\Gamma)$ so that 
$d_\Gamma(x,\psi\circ\phi(x))\leq k$ for each vertex $x$ of $\Gamma$ 
and similarly $d_\Lambda(y,\phi\circ\psi(y))\leq k$ for each 
vertex $y$ of $\Lambda$.  Graphs are {\sl quasi-isometric} if they 
are $k$-quasi-isometric for some integer $k>0$.  

We remark that the above definition is not the standard one; see for
example~\cite[I.8.14]{bh} for the standard definition of a
quasi-isometry between metric spaces.  We leave it as an exercise to
check that graphs $\Gamma$, $\Lambda$ are quasi-isometric as above if and
only if the metric spaces $(\Gamma,d_\Gamma)$ and
$(\Lambda,d_\Lambda)$ are quasi-isometric in the usual sense.

An edge loop of length $l$ in a (simplicial) graph $\Gamma$ is a
sequence $v_0,\ldots,v_l$ of vertices such that $v_0=v_l$ and
$\{v_{i-1},v_{i}\}$ is an edge for $1\leq i\leq l$.  For a graph
$\Gamma$ and an integer constant $l$, let $\Gamma_l$ denote the
2-complex whose 1-skeleton is the geometric realization of~$\Gamma$, 
with one 2-cell attached to each edge loop in $\Gamma$ of length 
strictly less than $l$.  An edge loop of length~$l$ in $\Gamma$ 
is said to be {\sl taut} if it is not null-homotopic
in $\Gamma_l$.  Bowditch's {\sl taut loop length spectrum} $H(\Gamma)$
for the graph $\Gamma$ is the set of lengths of taut loops.

We are interested in the 2-complex $\Gamma_l$ only to define taut 
loops: if $\Gamma'$ is any subcomplex with $\Gamma\subseteq 
\Gamma'\subseteq \Gamma_l$ so that the induced map on fundamental
groups $\pi_1(\Gamma')\rightarrow \pi_1(\Gamma_l)$ is an isomorphism, 
then an edge loop is taut if and only if it is not null-homotopic 
in $\Gamma'$.  

Bowditch defines subsets $H,H'\subseteq \nn$ to be {\sl $k$-related} 
if for all $l\geq k^2+2k+2$, whenever $l\in H$ then there is some $l'\in H'$ 
with $l/k\leq l'\leq lk$ and vice-versa.  He then proves 

\begin{lemma}\label{lemma:bowditch}
If (connected) graphs $\Gamma$ and $\Lambda$ are 
$k$-quasi-isometric, then $H(\Gamma)$ and $H(\Lambda)$ are
$k$-related. 
\end{lemma} 

In our terms, the lemmas that Bowditch uses to prove the above 
result are as follows.  

\begin{lemma} $H(\Gamma)$ is equal to the set of $l\in \nn$ for which 
the induced map of fundamental groups 
$\pi_1(\Gamma_l)\rightarrow \pi_1(\Gamma_{l+1})$ is not an 
isomorphism.  
\end{lemma} 

For any fixed $l$, let $i_{\Gamma,l}$ denote the inclusion of 
$\Gamma$ in the 2-complex $\Gamma_l$.  

\begin{lemma} 
If $\phi:\Gamma\rightarrow \Lambda$ and $\psi:\Lambda\rightarrow
\Gamma$ are $k$-Lipschitz maps whose restrictions to vertex sets 
define a $k$-quasi-isometry between $\Gamma$ and $\Lambda$, 
then for any $l\geq k^2+2k+2$ there are homotopies 
\[ i_{\Gamma,l}\circ \psi\circ \phi\simeq i_{\Gamma,l} 
\quad\hbox{and}\quad 
i_{\Lambda,l}\circ \phi\circ \psi\simeq i_{\Lambda,l}.\] 
\end{lemma} 

Next we review some material from~\cite{bb,ufp}.  A {\sl flag
complex} or {\sl clique complex} is a simplicial complex 
in which every finite set of mutually adjacent vertices 
spans a simplex.  For the remainder of this section, 
$L$ will denote a finite flag complex.  
Let $\ttt$ denote the circle $\rr/\zz$, 
viewed as a CW-complex with one vertex at $0+\zz\in \rr/\zz$ 
and one edge.  For a finite set $V$ let $\ttt^V$ denote
the product $\prod_{v\in V}\ttt_v$, where $\ttt_v$ 
denotes a copy of $\ttt$.  Non-empty subcomplexes of 
$\ttt^V$ are in bijective correspondence with simplicial 
complexes with vertex set contained in $V$.  If $L$ 
is a finite flag complex with vertex set $V$, let $\ttt_L$ 
denote the corresponding subcomplex of $\ttt^V$.  This 
complex is aspherical, and its fundamental group 
is the {\sl right-angled Artin group} $A_L$ associated to 
$L$, with generators corresponding to the vertices 
of $L$, subject only to the relations that $v$~and~$w$ 
commute whenever $\{v,w\}$ is an edge of $L$.  

The universal covering space $X_L$ of $\ttt_L$ has a 
natural cubical structure, and is a CAT(0) cubical 
complex.  The additive group structure in $\ttt=\rr/\zz$
defines a map $l:\ttt^V\rightarrow \ttt$, and hence 
a map $l_L:\ttt_L\rightarrow \ttt$.  Define $\widetilde \ttt_L$ 
to be the regular covering of $\ttt_L$ induced by 
pulling back the universal covering of $\ttt$ along $l_L$.  
The Bestvina-Brady group $BB_L$ is defined to be the 
fundamental group $\pi_1(\widetilde \ttt_L)$, or 
equivalently the kernel of the map $A_L\rightarrow \zz$ 
of fundamental groups induced by $l:\ttt_L\rightarrow \ttt$.  
Bestvina and Brady showed that many properties of $BB_L$ 
are determined by properties of $L$.  In the case when $L$ is  
acyclic but not simply connected, $BB_L$ is type
$FP$ but not finitely presented~\cite{bb}.  

Let $f_L:X_L\rightarrow \rr$ be the map of universal coverings 
induced by $l_L$.  This map has the following properties: if 
we identify each $n$-cube of $X_L$ with $[0,1]^n$ then its 
restriction to each $n$-cube is equal to an affine map; the 
image of each vertex of $X_L$ is an integer; the image of 
each $n$-cube of $X_L$ is an interval of length $n$.  We 
view $f_L$ as defining a height function on $X_L$.  There 
is a regular covering map 
$X_L\rightarrow \widetilde\ttt_L$, with covering group 
$BB_L$.  

In~\cite{ufp} this is generalized, under the extra 
assumption that $L$ be connected.  For each 
set $S\subseteq \zz$, a CAT(0) cubical complex $X_L^{(S)}$ 
is defined, together with a regular branched covering map 
$X_L^{(S)}\rightarrow \widetilde \ttt_L$, and the 
group $G_L(S)$ is by definition the covering group 
for this covering.  The only branch points of this 
covering are the vertices of $X_L^{(S)}$ whose height 
is not in $S$, and the stabilizer in $G_L(S)$ of each 
branch point is a subgroup isomorphic to the fundamental
group $\pi_1(L)$.  (In particular, the construction is 
non-trivial only when $L$ is not simply-connected.)  
If $S\subseteq T\subseteq \zz$, there 
is a regular branched covering map $X_L^{(S)}\rightarrow 
X_L^{(T)}$, branched only at vertices of height in 
$T-S$, and the branched covering $X_L^{(S)}\rightarrow 
\ttt_L$ factors through this.  If $S\subseteq T$ then 
there is a surjective group homomorphism 
$G_L(S)\rightarrow G_L(T)$, and the branched covering map 
$X_L^{(S)}\rightarrow X_L^{(T)}$ is equivariant for this 
homomorphism.  The group $G_L(\zz)$ is equal to $BB_L$.  

The height function on $X_L$ induces a $G_L(S)$-invariant height
function on $X_L^{(S)}$ for each $S\subseteq \zz$.  Since $\widetilde
\ttt_L$ has only one vertex of each integer height, the group $G_L(S)$
acts transitively on the vertices of $X_L^{(S)}$ of each height.  The
intersection of the 2-skeleton of $X_L^{(S)}$ and the 0-level set
(i.e., the points of height 0) is a simplicial graph $\Gamma$ whose
0-skeleton is the vertices of height 0.  Orbits of edges in $\Gamma$
correspond to $A_L$-orbits of squares in $X_L$, or equivalently to
edges of $L$.  If $0\in S$ then $G_L(S)$ acts freely on $\Gamma$, and
so $\Gamma$ can be identified with a simplicial Cayley graph for
$G_L(S)$.  This gives a natural choice of generators for $G_L(S)$ when
$0\in S$, in bijective correspondence with the directed edges of $L$.
Under the composite map $G_L(S)\rightarrow G_L(\zz)=BB_L\rightarrow
A_L$ the element corresponding to the directed edge from vertex $x$ to
vertex $y$ maps to the element $xy^{-1}$.  To give a presentation for
$G_L(S)$ with this generating set, we first fix a finite collection
$\Omega$ of directed loops in $L$ that normally generates $\pi_1(L)$.
In other words, if one attaches discs to $L$ along the loops in
$\Omega$, one obtains a simply-connected complex.  Three families of
relators occur in this presentation, which we call $P(L,\Omega)$:

\begin{itemize} 
\item{} (Edge relations) for each directed edge $a$ with opposite 
edge $\overline a$, the relation $a\overline{ a}=1$; 

\item{} (Triangle relations) for each directed triangle $(a,b,c)$ in $L$ 
the relations $abc=1$ and $a^{-1}b^{-1}c^{-1}=1$; 

\item{} (Long cycle relations) for each $n\in S-\{0\}$ and each $(a_1,\ldots,
a_l)\in \Omega$ the relation $a_1^na_2^n \cdots a_l^n=1$.  

\end{itemize}

Another crucial property of these presentations is that only the long
relations corresponding to $n\in S$ hold in $G_L(S)$: if
$(a_1,\ldots,a_l)\in \Omega$ is not null-homotopic in $L$, then
$a_1^na_2^n\cdots a_l^n\neq 1$ for $n\notin S\cup\{0\}$.

We close by giving some references for more general background material.  
For CAT(0) spaces we suggest~\cite{bh}, and for homological finiteness
conditions such as the $FP$ property we suggest~\cite{brown}.  Each of 
these topics is also covered briefly in the appendices to~\cite{davbook}, 
which is our recommended source for Coxeter groups.  

\section{Bounding word lengths by CAT(0) distances} 
\label{sec:three}

Our first task is to establish Lemma~\ref{lem:cat0}, as stated in the
Introduction.  Recall that a singular point of a map between CAT(0)
spaces is a point at which the map is not a local isometry.

\begin{proof} (Lemma~\ref{lem:cat0}).   
As in the statement, let $x,x'\in X$ be distinct points 
such that $f(x)=f(x')$, and 
suppose that the geodesic arc $\gamma$ from $x$ to $x'$ does not pass
through the singular set.  In this case, $f\circ \gamma$ is a locally 
geodesic arc in $Y$, whose end points are both equal to $f(x)$.  In a 
CAT(0) space any locally geodesic arc is a geodesic arc, and the 
unique geodesic arc from $f(x)$ to $f(x')=f(x)$ is the constant 
arc of length~0.  This contradiction shows that $\gamma$ must pass 
through the singular set.  The claim follows.  
\end{proof} 

\begin{lemma} \label{lem:height}
Let $L$ be a finite connected flag complex of dimension $d$.  For any
$S\subseteq \zz$ the distance from the 0-level set in the 
CAT(0) space $X=X_L^{(S)}$ to a vertex of height $n$ is $|n|/\sqrt{d+1}$.  
\end{lemma} 

\begin{proof} 
By symmetry it suffices to consider the case $n>0$.  Let 
$\gamma$ be a path starting at a vertex $v$ of height $n$, 
and moving at unit speed in $X$ to the 0-level set, and 
let $f:X\rightarrow \rr$ denote the height function on~$X$.   
Minimizing the length of $\gamma$ is equivalent to maximizing
the speed of descent, i.e., minimizing the derivative of 
$f\circ \gamma$.  

The initial direction of travel of the path $\gamma$ can be 
represented by a point of the link, $\link_X(v)$, of $v$ in 
$X$.  This is a simplicial complex in which each 
$m$-cube $C$ of $X$ that is incident on $v$ contributes 
one $(m-1)$-simplex, consisting of the unit tangent vectors 
at $v$ that point into $C$.  

If we identify an $m$-cube $C$ of $X$ 
with $[0,1]^m$, then $f$ restricted to $C$ is equal to 
$(t_1,\ldots,t_m)\mapsto t_1+t_2+\cdots+t_m+r$ for some integer~$r$. 
The gradient of $f$ on the cube $C$ is the vector $(1,1,\ldots,1)$, 
of length $\sqrt{m}$.  Thus any path $\gamma$ of fastest descent 
leaves $v$ travelling in the direction of the long diagonal of 
a cube $C$ of maximal dimension whose highest vertex is $v$.  

If $\sqrt{m}<n$ then the path will reach the unique lowest vertex $v'$ 
of $C$ before it reaches the 0-level set; at this vertex a new 
choice of cube $C'$ of maximal dimension with $v'$ as its highest 
vertex should be made.  

If $w$ is any vertex of $X$ then the cubes of $X$
that have $w$ as their highest vertex correspond to a subcomplex 
of $\link_X(w)$ called the descending or $\downw$-link, 
$\link^\downw_X(w)$.  Each descending link $\link^\downw_X(w)$ is 
isomorphic to either $L$ or its universal covering space 
$\widetilde L$, depending only on whether the height of 
$w$ lies in $S$.  In particular, each descending link has dimension 
equal to $d$, the dimension of $L$.  

It follows that we can always find at least one unit speed path
$\gamma$ starting at $v$ with constant rate of descent $\sqrt{d+1}$
and there is no path descending faster.  Hence the distance from $v$
to the 0-level set is $n/\sqrt{d+1}$ as claimed.
\end{proof}

For $S\subseteq \zz$, define $m(S):=\min\{ |n| : n\in S\}$.  

\begin{lemma} Suppose that $L$ is $d$-dimensional and that 
$0\in S\subseteq T\subseteq \zz$, and take the standard 
generating set for $G_L(S)$ and $G_L(T)$.  The word length 
of any non-identity element in the kernel of the map 
$G_L(S)\rightarrow G_L(T)$ is at least $m(T-S)\sqrt{2/(d+1)}$.  
\end{lemma} 

\begin{proof} 
The Cayley graph $\Gamma(S)$ for $G_L(S)$ is embedded in the 
0-level set in $X:=X_L^{(S)}$, and similarly $\Gamma(T)$ is embedded
in the 0-level set in $Y:=X_L^{(T)}$.  Moreover the branched covering
map $X\rightarrow Y$ induces the natural quotient 
map $\Gamma(S)\rightarrow \Gamma(T)$.  Let $v$ be a height 0 vertex
of $X$.  Each standard generator for $G_L(S)$ is represented 
by the diagonal of a square of $X$ so for any $g\in G_L(S)$
the triangle inequality tells us that the word length $l(g)$ 
satisfies $d_X(v,gv)\leq \sqrt{2}l(g)$.  Now $g$ is in the kernel 
of the map to $G_L(T)$ if and only if $gv$ and $v$ map to the same 
vertex of $Y$.  Singular points for the map $X\rightarrow 
Y$ are vertices $w$ whose heights lie in $T-S$, and by 
Lemma~\ref{lem:height} these have $d_X(v,w)\geq m(T-S)/\sqrt{d+1}$.  
By Lemma~\ref{lem:cat0} it follows that $d_X(v,gv)\geq 2m(T-S)/\sqrt{d+1}$ 
and hence $l(g)\geq d_X(v,gv)/\sqrt{2} \geq m(T-S)\sqrt{2/(d+1)}$.  
\end{proof} 

\section{Digression on convexity} 

The arguments used in the previous section can be used to show 
that the 0-level sets are very rarely convex or even quasi-convex.  
The material in this section is not needed for our main theorem.

\begin{corollary} 
The 0-level set in $X_L^{(S)}$ is convex if and only if $L$ is a 
single simplex.  In this case $X_L^{(S)}=X_L$ does not depend on $S$.  
\end{corollary} 

\begin{proof} 
If $L$ is a $d$-simplex then $X_L^{(S)}=X_L$ is a copy of $\rr^{d+1}$
and the 0-level set is an affine subspace.  For the converse, if $L$ 
is any flag complex other than a single simplex, then $L$ will contain 
at least two maximal simplices.  If $v$ is any vertex of $X_L^{(S)}$ 
of height one, the directions defined by the barycentres of these 
two maximal simplices give distinct geodesic paths from $v$ to the 
0-level set that are both locally distance minimizing, with end points
$x$, $x'$ of height~0.  Within the geodesic arc from $x$ to $x'$, at 
most one point can locally minimize distance to $v$; the assumption 
that this geodesic arc lies in the 0-level set leads to a
contradiction.  
\end{proof}

\begin{corollary} 
If either $L$ contains two simplices of maximal dimension $d$, 
or $\widetilde L$ does and $\zz-S$ is infinite, then 
the 0-level set in $X_L^{(S)}$ is not quasi-convex.  
\end{corollary} 

\begin{proof} 
We give only a sketch.  Let $\theta=\theta(d)$ be the angle 
in $\rr^{d+1}$ between the vector $(1,1,\ldots,1)$ and one 
of the coordinate hyperplanes.  Let $v$ be a vertex of 
height $N$ in $X:=X_L^{(S)}$.  If $L$ has a unique simplex 
of dimension $d$, then $N$ should be chosen in $\zz-S$, 
otherwise $N$ may be arbitrary.  
In this case there are two distance-minimizing 
geodesic paths from $v$ to the 0-level set, with end points 
$x$ and $x'$ as above, corresponding to leaving $v$ in the 
directions given by the barycentres of two distinct $d$-dimensional 
simplices of the descending link.  We view $x$, $x'$ and other 
points that depend on them as functions of~$N$.  
The angle at $v$ between these two geodesics is at least 
the constant $2\theta$.  The geodesic 
triangle with vertices $x$, $x'$ and $v$ is isosceles with 
angle at least $2\theta$ between the two equal sides.  The 
length of the equal sides is $N/\sqrt{d+1}$.  If 
$y$ is the midpoint of the geodesic arc from $x$ to $x'$, 
it follows that $d_X(v,y)\leq N\cos(\theta)/\sqrt{d+1}$.  
By increasing $N$ this distance can be made arbitrarily 
smaller than $N/\sqrt{d+1}$, the distance from $v$ to 
the 0-level set.  Hence for any $k$, there is an $N$ so 
that $y$ is not in the $k$-neighbourhood of the 0-level 
set.  Thus the 0-level set is not quasi-convex.  
\end{proof}

\section{Taut loop length spectra for $G_L(S)$} 

Throughout this section we fix a finite connected non-simply connected 
flag complex $L$.  For $S$ a subset of $\zz$ containing 0, let 
$\Gamma(S)$ denote the Cayley graph of $G_L(S)$ with respect to the 
standard generators.  We analyze 
the taut loop length spectrum $H(\Gamma(S))$.  As in~\cite{bhb}, it 
will be convenient to assume that elements of $S$ grow quickly, 
which we do as follows.  

Define $\alpha=\alpha(L)$ by $\alpha=\sqrt{2/(d+1)}$, where 
$d$ is the dimension 
of $L$.  For a finite set $\Omega$ of loops in $L$ that normally generates 
$\pi_1(L)$, let $\beta(L,\Omega)$ be the maximum of the lengths of 
the loops in $\Omega$, and define $\beta=\beta(L)$ to be the minimum value of 
$\beta(L,\Omega)$ over all such $\Omega$.  Choose an integer constant 
$C=C(L)$ so that $C> \beta/\alpha$ and $C\alpha > 3$.  For $F$ any subset 
of $\nn$, define $S(F)=\{0\}\cup\{ C^{2^n} : n\in F\}$.  With these 
definitions we prove an analogue of~\cite[Proposition~1]{bhb}. 

\begin{proposition}\label{symdiffprop} 
If $F,F'$ are subsets of $\nn$ so that $\Gamma(S(F))$ and $\Gamma(S(F'))$ 
are quasi-isometric then the symmetric difference of $F$ and $F'$ is finite. 
\end{proposition} 

Theorem~\ref{thm:main} follows immediately from this Proposition.  
To prove the Proposition we first describe $H(\Gamma(S(F)))$.  

\begin{theorem} \label{thm:aitchgamma}
For any $F\subseteq \nn$, the set $H(\Gamma(S(F)))$ is contained in 
the disjoint union 
$\{3\} \cup \bigcup_{n\in \nn} [\alpha C^{2^n},\beta C^{2^n}]$.  
The set $H(\Gamma(S(F)))\cap [\alpha C^{2^n},\beta C^{2^n}]$ is 
non-empty if and only if $n\in F$.  Also $3$ is in $H(\Gamma(S(F)))$ 
if and only if $L$ is not 1-dimensional.  
\end{theorem} 

\begin{proof} 
The choice of $C$ ensures that $3< \alpha C^{2^0}$ and that 
for all $n$, $\beta C^{2^n} < \alpha C^{2^{n+1}}$, which implies 
that the union is disjoint.  Since $\Gamma=\Gamma(S(F))$ is a simplicial
Cayley graph, the edge relations $a\overline{a}=1$ do not 
contribute to $H(\Gamma)$, but the triangle relations imply 
that $3\in H(\Gamma)$ whenever $L$ has dimension at least two.  
If $F=\emptyset$ then the presentation $P(L,\Omega)$ contains 
only relations of length at most 3, so $H(\Gamma)$ is either 
empty if $L$ is 1-dimensional or is equal to $\{3\}$ otherwise.  

It remains to establish three statements
\begin{itemize} 
\item{} If $n\in F$ then $H(\Gamma(S(F)))\cap [\alpha C^{2^n}, 
\beta C^{2^n}] \neq \emptyset$; 

\item{} If $n\notin F$ then $H(\Gamma(S(F)))\cap [\alpha C^{2^n}, 
\beta C^{2^n}] =\emptyset $; 

\item{} If $k>3$ and $k\notin \bigcup_{n\in \nn} [\alpha C^{2^n}, 
\beta C^{2^n}]$ then $k\notin H(\Gamma(S(F)))$.  

\end{itemize} 

The second and third of these statements can be grouped together 
into a single fourth statement: 

\begin{itemize} 
\item{} If $k>3$ and $k\notin \bigcup_{n\in F} [\alpha C^{2^n}, 
\beta C^{2^n}]$ then $k\notin H(\Gamma(S(F)))$.  

\end{itemize} 

For the first statement, let $F'=F-\{n\}$, and consider the 
covering map $\Gamma(S(F'))\rightarrow \Gamma(S(F))$.  The 
group $G_L(S(F'))$ acts freely on $\Gamma(S(F'))$,  
so we may attach free orbits of 
2-cells to $\Gamma(S(F'))$ to make a simply-connected 
Cayley 2-complex $\Delta$.  Now let $K$ be 
the kernel of the map $G_L(S(F'))\rightarrow G_L(S(F))$, 
or equivalently the covering group for the regular 
covering $\Gamma(S(F'))\rightarrow \Gamma(S(F))$.   
The quotient $\Delta/K$ is a 2-complex with 1-skeleton 
the graph $\Gamma(S(F))$ and fundamental group $K$.  We know 
that any non-identity element of $K$ has word length at least 
$\alpha C^{2^n}$ and that there is a non-identity element of 
$K$ of word length $\beta C^{2^n}$.  The shortest non-identity 
element of $K$ defines a loop in $\Gamma(S(F))\subseteq \Delta/K$ 
that must be taut, since it is not null-homotopic in $\Delta/K$ 
whereas every strictly shorter loop in $\Gamma(S(F))$ is 
null-homotopic in $\Delta/K$.  

It remains to prove the fourth statement.  Fix an integer 
$k>3$ that is not an element of 
$\bigcup_{n\in F} [\alpha C^{2^n},\beta C^{2^n}]$.  
Choose $n\in \nn$ maximal so that $\beta C^{2^n} <k$, if 
such $n$ exists, and define $n:=-1$ in the case when 
$3<k< \alpha C$.  Now let $F':= F\cap [0,n]$, where by 
definition $[0,-1]=\emptyset$.  Once again, consider 
the covering map $\Gamma(S(F'))\rightarrow \Gamma(S(F))$.  
Since every relator in the presentation for $G_L(S(F'))$ 
has length at most $\beta C^{2^n}$, we may build a Cayley
2-complex $\Delta$ with 1-skeleton $\Gamma(S(F'))$ in 
which each 2-cell is attached to a loop of length at 
most $\beta C^{2^n}$.  Now suppose that $\gamma$ is a 
loop of length $k$ in $\Gamma(S(F))$.  If $\gamma$ 
lifts to a loop in $\Gamma(S(F'))$ then it cannot be 
taut, since every loop in $\Gamma(S(F'))$ is null-homotopic in 
$\Delta$.  If on the other hand $\gamma$ lifts to a non-closed 
path in $\Gamma(S(F'))$ then it corresponds to a non-identity
element of the kernel of the map $G_L(S(F'))\rightarrow 
G_L(S(F))$ of word length at most~$k$.  But the shortest 
element in the kernel of this map has length at least  
$\alpha C^{2^m}$, where $m$ is the least element of 
$S(F)-S(F')$.  By choice of $n$, we have that $k\leq \beta C^{2^m}$, 
and by hypothesis $k\notin [\alpha C^{2^m},\beta C^{2^m}]$.  
This contradiction shows that the loop $\gamma$ cannot be 
taut.  
\end{proof}

\begin{proof} (Proposition~\ref{symdiffprop}).  
For any $l \in [\alpha C^{2^n},\beta C^{2^n}]$ and $l'\in 
[\alpha C^{2^{n+m}},\beta C^{2^{n+m}}]$ for some $m>0$, we have that 
$l'/l \geq C^{2^n-1}$.  By Lemma~\ref{lemma:bowditch}, since 
$\Gamma(S(F))$ and $\Gamma(S(F'))$ are quasi-isometric, 
$H(\Gamma(S(F)))$ and $H(\Gamma(S(F')))$ are $k$-related.  
Now if $n$ is in the symmetric 
difference of $F$ and $F'$ we see that $C^{2^n-1} < k$.  
\end{proof} 

\section{Poincar\'e duality groups and semi-direct products}

A {\sl right-angled Coxeter group} is a group $W$ admitting a
presentation in which the only relators are that each generator 
has order two and that certain pairs of generators commute.  
A simplicial graph $K$ gives rise to a right-angled Coxeter 
group $W_K$, with generators the set $V(K)$ of vertices of $K$ 
and as commuting pairs the ends of each element of the edge 
set $E(K)$.  As we consider cases when $K$ is infinite, we 
start with a well-known lemma that will help us reduce to the finite 
case.  

\begin{lemma} \label{lemma:retract}
Let $K'$ be any full subgraph of the simplicial graph 
$K$.  The Coxeter group $W'=W_{K'}$ is a retract of 
the Coxeter group $W=W_K$.  
\end{lemma} 

\begin{proof} 
Let $V:=V(K)$ and $V':=V(K')$.  The inclusion of $V'$ into 
$V$ extends to define a group homomorphism $i:W'\rightarrow W$.   
The function $\pi: V\rightarrow V'\cup\{1\}$ defined by 
$\pi(v')=v'$ for $v'\in V'$ and $\pi(v)=1$ for $v\notin V'$ 
extends to a group homomorphism $\pi:W\rightarrow W'$ and 
the composite $\pi\circ i:W'\rightarrow W'$ is the identity.  
\end{proof} 

Now suppose that a group $G$ acts by automorphisms on the 
graph $K$.  This induces an action of $G$ on $W_K$ by 
automorphisms, permuting the given generators for $W_K$, 
and so there is a semidirect product group $J=W_K\semi G$.  
Identify $G$ with its image inside the semidirect product $J$.  
A choice of generating set for $G$ together with a choice 
of $G$-orbit representatives in $V(K)$ gives rise 
to a generating set for $J$.  

Now suppose that $S\mapsto G(S)$ is a functor from the 
category of subsets of $\zz$ with inclusions as morphisms 
to the category of finitely generated groups and surjective
homomorphisms; for example $S\mapsto G_L(S)$ is such a 
functor for any connected finite flag complex $L$.  
Suppose further that $G(\emptyset)$ acts freely 
cocompactly on a (simplicial) graph $K(\emptyset)$ in such a way that 
any two vertices in the same $G(\emptyset)$-orbit are at 
edge path distance at least four.  For $S\subseteq \zz$, 
define $K(S)$ to be the quotient of $K(\emptyset)$ by the
kernel of the map $G(\emptyset)\rightarrow G(S)$, so that 
$G(S)$ acts freely cocompactly on the graph $K(S)$.  

For $S\subseteq \zz$, define $J(S)$ to be the semidirect product
$W_{K(S)}\semi G(S)$.  Then $S\mapsto J(S)$ is another functor from
subsets of $\zz$ and inclusions to finitely generated groups and
surjective group homomorphisms.  Fix a finite generating set for
$J(\emptyset)$ consisting of a finite generating set for
$G(\emptyset)$ and a set $V'$ of $G(\emptyset)$-orbit representatives in
$V(K(\emptyset))$.  As generating set for $J(S)$, take the image of
our given generating set for $J(\emptyset)$, and as generating set 
for $G(S)$ take the image of our given generating set for
$G(\emptyset)$.  For each $S$, the generating set for $G(S)$ is 
a subset of the generating set for $J(S)$, and its complement 
consists of generators that are in the kernel of the map
$J(S)\rightarrow G(S)$.  

It will be useful to have a presentation for $J(S)$ in terms of 
our generating set.  Since $G(\emptyset)$ acts freely on the 
graph $K(\emptyset)$, the Coxeter relators between all of the 
generators for $W_{K(\emptyset)}$ are consequences of a finite 
set of relators indexed by the orbit representatives of vertices
and edges in $K(\emptyset)$.  To describe these relations, we 
choose a set $E'$ of $G(\emptyset)$-orbit representatives of 
edges in $K(\emptyset)$, in such a way that each $e\in E'$ 
is incident on at least one $v\in V'$.  For each $u\in
V(K(\emptyset))$, let $g_u\in G(\emptyset)$ be the unique
element such that $g_u.u\in V'$.  Now define an integer $N_1$ 
by 
\[ N_1:= \max \{l(g_u) \colon \hbox{$u$ is incident on some edge in
  $E'$}\},\]  
where $l(g)$ denotes the word length of $g\in G(\emptyset)$.  
The relations in our presentation for $J(S)$ are of the following 
kinds:
\begin{itemize} 
\item{}
$v^2$ for each $v\in V'$; 
\item{}
relators $(vg_uug_u^{-1})^2$, where $e\in E'$, $u$ and $v$ are the 
vertices incident on $e$ and $v\in V'$; 
\item{}
the relators in a presentation for $G(S)$.  
\end{itemize}
Note that $J(S)$ is finitely presented whenever $G(S)$ is, and that 
the relators of the second kind are of length at most $4N_1+4$.  

In the theorem below we write $l_{J(S)}$ and $l_{G(S)}$ for 
the word length with respect to these generating sets.  

\begin{theorem} \label{thm:semiker}
Take notation and hypotheses as in the paragraphs above, and 
define $N:=2N_1$.  For all $S\subseteq T\subseteq \zz$, 
if $w\in J(S)$ is in the kernel of $J(S)\rightarrow J(T)$ 
and $l_{J(S)}(w)>N$ then there is $g\in G(S)-\{1\}$ with 
$l_{G(S)}(g)\leq l_{J(S)}(w)$ so that $g$ is 
in the kernel of $G(S)\rightarrow G(T)$.  
\end{theorem}

\begin{proof} 
In the case when $w$ is not in the kernel of the map $J(S)\rightarrow
G(S)$, we may take $g=\overline w$, the image of $w$, since 
this element is in the kernel of $G(S)\rightarrow G(T)$ and 
$l_{G(S)}(g)\leq l_{J(S)}(w)$.

Before starting the remaining (more difficult) case, we 
recall Tits' solution to the word problem for a right-angled Coxeter 
group~\cite[Theorem~3.4.2]{davbook}.  (This is usually only stated 
for the finitely generated case, but the general case follows via
Lemma~\ref{lemma:retract}.)  
If $w=v_1v_2\cdots v_l$ is a word in the standard generators 
for a right-angled Coxeter group that represents the identity, then 
$w$ can be reduced to the trivial word using some sequence of moves 
of two types: 
\begin{itemize}
\item{}
if $v,v'$ are Coxeter generators that commute, replace 
the subword $vv'$ by~$v'v$; 

\item{}
replace a subword $vv$ by the trivial subword.  
\end{itemize} 

The kernel of the map $J(S)\rightarrow G(S)$ is the right-angled 
Coxeter group $W(S):=W_{K(S)}$.  Let $w$ be in the kernel of this 
map as well as in the kernel of the map $J(S)\rightarrow J(T)$.  
Pick a shortest word in the generators for $J(S)$ representing 
$w$, and write this word in the form $w=h_0v_1h_1v_2h_2\cdots h_{n-1}v_nh_n$, 
where each $v_i\in V'$, each $h_i\in G(S)$, and so that 
$l_{J(S)}(w)=n+\sum_{i=0}^n l_{G(S)}(h_i)$.  
Now define $g_i= h_0h_1\cdots h_{i-1}$ for $1\leq i\leq n$.  
We have that 
\[w=h_0v_1h_1v_2h_2\cdots h_{n-1}v_nh_n = (g_1v_1g_1^{-1})(g_2v_2g_2^{-1})
\cdots (g_nv_ng_n^{-1}).\]
This second expression for $w$ will not be reduced in general, but 
each subword $g_iv_ig_i^{-1}$ is equal to one of the standard 
Coxeter generators for the subgroup $W(S)$.  By hypothesis $w$ 
is non-trivial in $W(S)$ but is in the kernel of the 
map $W(S)\rightarrow W(T)$.  Hence there must be a Tits move that 
can be applied to the image of this expression in $W(T)$ that 
cannot be applied to the same expression in $W(S)$.  

If there is a Tits move of the second type that can be applied 
in $W(T)$ but not in $W(S)$, this implies that there exist $i$ 
and $j$ with $1\leq i < j\leq n$ so that $g_i=g_j\in G(T)$, but 
$g_i\neq g_j\in G(S)$.  If on 
the other hand there is a Tits move of the first type that can 
be applied in $W(T)$ but not in $W(S)$, there exist $1\leq i<j\leq n$ 
so that $g_i=g_jg_u\in G(T)$ but $g_i\neq g_jg_u\in G(S)$, where 
$g_u$ is one of the elements that takes a vertex of some edge 
in $E'$ to a vertex in $V'$, and so  
by definition of $N_1$, $l_{G(S)}(g_u)\leq l_{G(\emptyset)}(g_u)\leq N_1$.  

Define an element of $J(S)$ by $w':=v_ih_iv_{i+1}\cdots h_{j-1}v_j$,
where $i$~and~$j$ are as above.  Since the expression
$w=h_0v_1h_1v_2h_2\cdots h_nv_n$ is of minimal length in terms of our
generators for $J(S)$, the length of the defining expression for $w'$
is also minimal.  But $l_{J(S)}(w')$ is greater than or equal to the
length of its image in $G(S)$, $h_ih_{i+1}\cdots h_{j-1}=
g_i^{-1}g_j$.  Thus $l_{J(S)}(w')\geq l_{G(S)}(g_i^{-1}g_j)$.
Depending on which sort of Tits move was applied, $g_i^{-1}g_j$ is either
a non-trivial element of the kernel of the map $G(S)\rightarrow G(T)$
or differs from such an element by some $g_u$.  Let $g$ be this
element of the kernel, and note 
that $l_{G(S)}(g)\leq l_{G(S)}(g_i^{-1}g_j)+N_1$.
Since $w$ maps to the identity element of $G(S)$ and the image in
$G(S)$ of $w'$ is a path whose endpoints are at distance at least
$l_{G(S)}(g_i^{-1}g_j)$, we see that $l_{J(S)}(w)\geq
2l_{G(S)}(g_i^{-1}g_j)$.   Putting
these inequalities together we obtain 
$l_{G(S)}(g)\leq l_{J(S)}(w)/2 +N_1$.  Since we may assume that 
$l_{J(S)}(w)>N= 2N_1$, this implies that $l_{G(S)}(g)\leq l_{J(S)}(w)$ as
required. 
\end{proof} 

Now we specialize to the case of interest; the case when 
$G(S)$ is the group $G_L(S)$ for some finite flag complex~$L$
that is connected but not simply connected.  
In this case, Theorem~\ref{thm:semiker} allows us to prove
an analogue of Theorem~\ref{thm:aitchgamma} for the semidirect product
$J(S):=W_{K(S)}\semi G_L(S)$.  We define constants $\alpha=\alpha(L)$, 
$\beta:=\beta(L)$ and $C:=C(L)$ as in the previous section, and for 
$F\subseteq\nn$ we define $S(F)\subseteq\nn$ as before.  Finally, we
define $M:=4N_1+4$, which depends on both $L$ and on the action of 
$G(\emptyset)$ on $K(\emptyset)$, and we denote by $\Lambda(S(F))$ 
the Cayley graph $\Gamma(J(S(F)))$.  

\begin{theorem} 
For any $F\subseteq \nn$, the set $H(\Lambda(S(F)))$ is contained in 
the union $[0,M]\cup \bigcup_{n\in \nn}[\alpha C^{2^n},\beta C^{2^n}]$.  
If $\alpha C^{2^n}>M$, then $H(\Lambda(S(F)))\cap [\alpha C^{2^n},
\beta C^{2^n}]$ is non-empty if and only if $n\in F$.  
\end{theorem} 

\begin{proof} 
Since the presentation for $G_L(S(\emptyset))$ has only relators 
of length at most~3, we see that our presentation for 
$J(S(\emptyset))$ consists of relators of length at most $M$.  
This verifies the claim in the case when $F=\emptyset$.  
As in the proof of Theorem~\ref{thm:aitchgamma}, it suffices 
to verify two claims: 
\begin{itemize} 
\item{}If $n\in F$ and $\alpha C^{2^n}>M$ then $H(\Lambda(S(F)))\cap 
[\alpha C^{2^n},\beta C^{2^n}] \neq \emptyset$; 
\item{}If $k>M$ and $k\notin \bigcup_{n\in F} [\alpha C^{2^n},\beta 
C^{2^n}]$, then $k\notin H(\Lambda(S(F)))$.  
\end{itemize}  
These statements can be verified exactly as in Theorem~\ref{thm:aitchgamma}.  
For the first, we consider $F':=F-\{n\}$ and look at the covering map 
$\Lambda(S(F'))\rightarrow \Lambda(S(F))$.  Attach free $J(S)$-orbits 
of 2-cells to $\Lambda(S(F'))$ to make a simply-connected Cayley complex
$\Delta$, and let $K'$ be the kernel of the map $J(S(F'))\rightarrow
J(S(F))$.  The quotient $\Delta/K'$ has 1-skeleton $\Lambda(S(F))$ and 
fundamental group $K'$.  As before, we know that there is an element 
of $K'\cap G(S(F'))$ of length $\beta C^{2^n}$ and that any
non-identity element of this subgroup has length at least $\alpha
C^{2^n}$.  Since $M=4N_1+4> N=2N_1$, Theorem~\ref{thm:semiker} tells 
us that the word length of any non-identity element of $K'$ is also 
at least $\alpha C^{2^n}$.  Now the shortest non-identity element of
$K'$ defines a taut loop in the required range.  

For the second statement, given such a $k$, take $n$ maximal so that 
$\beta C^{2^n} <k$, let $F':=F\cap [0,n]$, and consider the covering 
map $\Lambda(S(F'))\rightarrow \Lambda(S(F))$.  As before, we can 
build a Cayley 2-complex $\Delta$ with 1-skeleton $\Lambda(S(F'))$ 
in which each 2-cell is attached to a loop of length at most 
$\max\{M,\beta C^{2^n}\}$.  If $\gamma$ is a loop in $\Lambda(S(F))$ 
of length $k$, then either $\gamma$ is not taut, or the lift of 
$\gamma$ to $\Lambda(S(F'))$ is a non-closed path.  In the second case  
one obtains a non-identity element in the kernel of the map $J(S(F'))
\rightarrow J(S(F))$ of length at most $k$.  Since $k>M>N$,
Theorem~\ref{thm:semiker} tells us that there is a non-identity 
element of length at most $k$ in the kernel of the map 
$G(S(F'))\rightarrow G(S(F))$, which cannot happen.  This 
contradiction shows that $\gamma$ cannot be taut.  
\end{proof} 

The next two corollaries follow easily by the same proofs as 
in the previous section.  

\begin{corollary} 
If $F$, $F'$ are subsets of $\nn$ so that $\Lambda(S(F))$ and 
$\Lambda(S(F'))$ are quasi-isometric, the symmetric difference 
of $F$ and $F'$ is finite.
\end{corollary} 

\begin{corollary} \label{cor:jess} 
For any $L$ that is not simply-connected, and any graph $K(\emptyset)$ 
with a free $G_L(\emptyset)$-action, there are continuously many 
quasi-isometry classes of groups $J(S)$.  
\end{corollary} 

The reason why Corollary~\ref{cor:jess} is of value concerns the use of the 
Davis trick~\cite{davbook} to construct non-finitely presented 
Poincar\'e duality groups, as described in~\cite[Sec.~18]{ufp}.  
The starting point is a 2-complex $L$ for which each $G_L(S)$ 
is type $FP$; for this group there is a finite 2-complex that 
is an Eilenberg-Mac~Lane space $K(G_L(\emptyset),1)$.  For 
any $n\geq 4$, one can find a compact $n$-manifold $V$ with 
boundary that is also a $K(G_L(\emptyset),1)$.  Now let $K$ 
be the 1-skeleton of the barycentric subdivision of a 
triangulation of the boundary of $V$, and let $K(\emptyset)$ 
be the 1-skeleton of the induced triangulation of the boundary 
of the universal cover of $V$, with $G_L(\emptyset)$ acting 
via deck transformations.  For this choice of $K(\emptyset)$,
the group $J(\emptyset)$ contains a finite-index torsion-free 
subgroup $J'$ that is the fundamental group of a closed aspherical 
$n$-manifold $M$, and such that regular covering $M(S)$ of $M$ 
with fundamental group the kernel of $J'\rightarrow J(S)$ is acyclic 
for each $S\subseteq \zz$.  One deduces that each $J(S)$ contains a 
finite-index torsion-free subgroup that is a Poincar\'e 
duality group of dimension~$n$.  Since the inclusion of a 
finite-index subgroup is always a quasi-isometry, Corollary~\ref{cor:jess} 
implies Corollary~\ref{cor:manypdn} as stated in the introduction.  

It remains to prove Corollary~\ref{cor:manymani} from 
the introduction.  This follows from the above discussion 
by the Schwarz-Milnor Lemma~\cite[I.8.19]{bh}, which tells us that 
the acyclic covering manifold $M(S)$ of $M$ is quasi-isometric to
the group~$J(S)$.

\end{document}